\documentclass[a4paper,12pt]{amsart}
\usepackage{graphicx,amsmath,amsfonts,amssymb,amsthm,mathrsfs}
\usepackage[a4paper,margin=1in]{geometry}
\usepackage{mathtools}
\usepackage{xcolor}
\usepackage{geometry}
\DeclareMathOperator{\divv}{div}

\renewcommand{\mod}{\mathop{\mathrm{mod}}}

\newcommand{\cK}{\mathcal{K}}

\newcommand{\cF}{\mathcal{F}}

\newcommand{\cC}{\mathcal{C}}
\newcommand{\cL}{\mathcal{L}}
\newcommand{\Div}{\mathrm{div}}

\numberwithin{equation}{section}

\theoremstyle{plain}
\newtheorem{theorem}{Theorem}[section]
\newtheorem{lem}[theorem]{Lemma}

\newtheorem{prop}[theorem]{Proposition}
\newtheorem{example}[theorem]{Example}

\newtheorem{definition}[theorem]{Definition}

\newtheorem{remark}[theorem]{Remark}

\global\long\def\floor#1{\left\lfloor #1\right\rfloor }%

\begin{document}
	\title[Construction of Non-Special Divisors on Kummer Covers]{Construction of Non-Special Divisors on Kummer Covers with
		Arbitrary Ramification for LCP Codes}
	
	\author{Adler Marques, Yuri da Silva, and Saeed Tafazolian
	}
	\date{}
	
	\address{Universidade Federal do Rio de Janeiro --	Instituto de Matemática,	Cidade Universitária,	CEP 21941-909, Rio de Janeiro, Brazil}
	\email{adler@im.ufrj.br}
	
	\address{Departamento de Matem\'{a}tica - Instituto de Matem\'{a}tica, Estat\'{i}stica e Computação Cient\'{i}fica
		(IMECC) - Universidade Estadual de Campinas (UNICAMP), Rua S\'{e}rgio Buarque de Holanda, 651, Cidade Universit\'{a}ria,  Zeferino Vaz, Campinas, SP 13083-859, Brazil}

	\email{y225979@dac.unicamp.br}
	\email{saeed@unicamp.br}
	

	
	\begin{abstract}
		Linear Complementary Pairs (LCP) of algebraic geometry (AG) codes 
		offer strong resistance against side-channel and fault-injection 
		attacks, but their construction depends critically on the explicit 
		identification of non-special divisors of degree $g$ and $g-1$. 
		Existing constructions are restricted to Kummer extensions where 
		divisors are supported exclusively on totally ramified places, 
		significantly limiting the range of applicable function fields and codes. 
		We remove this restriction by developing a framework for general 
		Kummer extensions $y^m = \prod_{i=1}^r (x-\alpha_i)^{\lambda_i}$ 
		over finite fields with arbitrary ramification. Using Galois group 
		actions and invariant divisor techniques, we establish necessary and 
		sufficient conditions for non-speciality with no constraint on the 
		support, yielding explicit constructions where previous methods fail. 
		Our approach replaces the computationally intensive Weierstrass 
		semigroup machinery with a more direct and efficient framework. 
		As an application, we construct new explicit families of LCP AG codes 
		with determined parameters $[n,k,d]$, covering three ramification 
		regimes. The resulting codes meet or approach the Goppa designed 
		distance, offering greater flexibility for cryptographic applications.
	\end{abstract}

	\maketitle
	
	

	\section{Introduction}
	
	Linear Complementary Dual (LCD) codes and, more generally, Linear Complementary Pairs (LCP) of codes have become fundamental tools in modern cryptographic implementations due to their strong resistance against Side-Channel Attacks (SCA) and Fault Injection Attacks (FIA) \cite{Carlet2018}. Since their introduction by Massey \cite{Massey1992}, the construction of such codes with good parameters has remained an important problem in coding theory and cryptographic engineering.
	
	Algebraic Geometry (AG) codes, introduced by Goppa \cite{Goppa1981}, provide a powerful framework for constructing long linear codes with excellent asymptotic and explicit parameters. In particular, AG codes arising from algebraic function fields often achieve or approach classical bounds such as the Tsfasman--Vladut--Zink bound \cite{Tsfasman1982}. The construction of LCP AG codes is closely related to the existence of non-special divisors of degree $g$ and $g-1$, which is a delicate problem in the theory of algebraic function fields.
	
	Recent developments have focused on Kummer extensions of the form
	\[
	y^m = f(x),
	\]
	where $f(x)$ is not necessarily separable; see for instance \cite{Moreno2024},~\cite{4author} and \cite{Mendoza2026}. This represents a significant generalization at the level of the underlying function fields. However, despite this generality, existing constructions of divisors and AG codes in these works remain essentially restricted to \emph{totally ramified places}. In other words, even in non-separable settings, only places with full ramification are used in the construction of divisors. This restriction significantly reduces the set of available rational places and limits the flexibility of the resulting code families.
	
	In this work, we remove this limitation and develop a framework that allows the construction of non-special divisors supported on places with \emph{arbitrary ramification behavior}. This significantly enlarges the geometric and arithmetic range of admissible divisors, leading to a broader and more flexible class of AG code constructions.
	
	A second key contribution of this paper lies in the methodology. Existing approaches rely heavily on combinatorial tools such as generalized Weierstrass semigroups, which often lead to complicated and computationally intensive analyses. In contrast, our approach is more \emph{intrinsic and operational} (see Example \ref{Ex3.7}); it is based on the action of the Galois group on divisors together with systematic restriction arguments. This provides a more direct and computationally effective method for constructing non-special divisors.
	
	A central feature of our work is that it produces explicit examples of non-special divisors involving non–totally ramified places, where previous methods do not apply. This demonstrates that our framework strictly extends the scope of existing constructions in the literature.
	
	As an application, we construct new explicit families of LCP AG codes. For each construction, we compute the parameters $[n,k,d]$ and show that the resulting codes achieve strong minimum distance properties, often meeting or approaching the Goppa designed distance. Several explicit examples are provided to illustrate the effectiveness and flexibility of the proposed method.
	
	
	The main contributions of this paper are summarized as follows:
	\begin{itemize}
		\item \textbf{General Existence Criteria:} Establishing necessary and sufficient conditions for the existence of Galois-invariant non-special divisors of degree \(g\) in Kummer extensions with arbitrary ramification.
		\item \textbf{Methodological Advancement:} Introducing a constructive framework based on Galois actions that replaces the computationally intensive Weierstrass semigroup approach and removes the restriction to totally ramified places.
		\item \textbf{Explicit LCP AG Codes:} Constructing new families of Linear Complementary Pair (LCP) codes with determined parameters $[n, k, d]$ that extend the scope of existing literature.
		\item \textbf{Comparative Validation:} Providing explicit examples where our approach succeeds in yielding valid constructions in scenarios where previous methods fail.
	\end{itemize}

	The remainder of this paper is organized as follows. Section 2 reviews the necessary background on algebraic function fields and AG codes. Section 3 constitutes the core of this paper. In this section, we both state and prove the main theorem, providing a rigorous theoretical foundation for our approach. Furthermore, we explicitly compute divisor families for several specific classes of Kummer function fields where the places are non-totally ramified. Finally, we demonstrate that our proposed criteria offer broader applicability and are significantly more efficient than existing methodologies, even when restricted to the classical cases of totally ramified places. Section 4 applies the theoretical results derived in the previous section to the construction of Linear Complementary Pair (LCP) codes. By utilizing the non-special divisors constructed in Section 3, we develop a systematic framework for generating LCP codes across several families of function fields. A key feature of this construction is the inclusion of places that are not totally ramified, which significantly expands the variety of available codes and surpasses the limitations found in existing literature. Furthermore, we provide explicit examples to illustrate the parameters and performance of these new code families.
	\section{Preliminaries}
	
	In this section, we provide a comprehensive review of the theoretical foundations required for the development of our results. We cover the essentials of algebraic function fields, the construction of algebraic geometry (AG) codes, the properties of non-special divisors, and the specific arithmetic structure of general Kummer extensions. Throughout this paper, $\mathbb{F}_q$ denotes a finite field of cardinality $q$, where $q$ is a power of a prime $p$.
	
	\subsection{Algebraic Function Fields and Divisors}
	
	Let $\cF/\mathbb{F}_q$ be an algebraic function field of genus $g$. We denote by $\mathbb{P}_F$ the set of all places of $\cF$. For a place $P \in \mathbb{P}_F$, $v_P$ denotes the discrete valuation associated with $P$, and $\mathcal{O}_P$ is the corresponding valuation ring. A divisor $G$ is a formal sum $G = \sum_{P \in \mathbb{P}_F} n_P P$, where $n_P \in \mathbb{Z}$ and only finitely many $n_P$ are non-zero. The degree of $G$ is defined as $\deg(G) = \sum n_P \deg(P)$.
	
	The Riemann-Roch space associated with a divisor $G$ is defined as:
	\[ \mathscr{L}(G) = \{ f \in \cF \setminus \{0\} : \divv(f) + G \geq 0 \} \cup \{0\}. \]
	According to the Riemann-Roch Theorem \cite{Stichtenoth2009}, the dimension $\ell(G) = \dim_{\mathbb{F}_q} \mathcal{L}(G)$ satisfies:
	\[ \ell(G) = \deg(G) + 1 - g + i(G), \]
	where $i(G) = \ell(K-G)$ is the index of speciality and $K$ is a canonical divisor of $\cF$.
	
	\begin{definition}
		A divisor $A$ in a function field $\cF/\mathbb{F}_q$ of genus $g$ is called \textbf{non-special} if its index of speciality $i(A) = 0$. 
	\end{definition}
	Let $A \in \operatorname{Div}(\cF)$. By the Riemann-Roch Theorem,  we have
	\begin{itemize}
		\item  $A$ is a non-special divisor of $g-1$ if and only if $\ell(A)=0$; 
		\item $A$ is a non-special divisor of $g$ if and only if $\ell(A) = 1$.
	\end{itemize}

	In this section, we investigate the fundamental properties that govern the relationship between non-special divisors of adjacent degrees. The identification of such divisors is a prerequisite for several applications in the theory of algebraic function fields, particularly where the dimension of the Riemann-Roch space must be precisely determined. The following results provide a systematic method to transition between non-speciality in degrees $g$ and $g-1$.
	
	\begin{prop} \label{prop:transition_nonspecial}
		Let $\cF/\mathbb{F}_q$ be a function field of genus $g$. The following properties characterize the relationship between non-special divisors of adjacent degrees:
		\begin{enumerate}
			\item \cite[Lemma~3]{Ballet2006} Let $A$ be an effective non-special divisor of degree $g$. For any rational place $P$ of $\cF$ that is not contained in the support of $A$, the divisor $A - P$ is a non-special divisor of degree $g-1$.
			
			\item \cite[Proposition~1.6.11]{Stichtenoth2009} Conversely, let $A$ and  $B$ be two divisor of $\cF$. If $A \le B$ and $A$ is non-special, then $B$ is also a non-special divisor. 
		\end{enumerate}
	\end{prop}

	\subsection{Structural Properties of General Kummer Extensions}
	
	We now focus on the core object of our study: the general Kummer extension. Let $\cF = \mathbb{F}_q(x, y)$ be the function field defined by:
	\begin{equation} \label{eq:kummer_full}
		y^m = f(x) = a \prod_{i=1}^r (x-\alpha_i)^{\lambda_i}
	\end{equation}
	where $a \in \mathbb{F}_q^*$, $m \geq 2$ with $\gcd(m, p) = 1$, and $\alpha_1, \dots, \alpha_r$ are distinct elements in $\mathbb{F}_q$. The integers $\lambda_i$ satisfy $1 \leq \lambda_i < m$. Let $\cL = \mathbb{F}_q(x)$ be the rational subfield.
	
	The ramification in $\cF/\cL$ is governed by the exponents $\lambda_i$. For each $i \in \{1, \dots, r\}$, let $P_i$ be the rational place of $\cL$ corresponding to $(x-\alpha_i)$. The following fundamental theorem provides the essential decomposition of these places, which we extend to include non-totally ramified cases \cite{Stichtenoth2009}.
	
	\begin{theorem}\label{Th2.3}
		For the extension $\cF/\cL$ defined in \eqref{eq:kummer_full}, the following holds:
		\begin{enumerate}
			\item A rational place $P \in \mathbb{P}_L$ ramifies in $\cF$ if and only if $P \in \{P_1, \dots, P_r\}$ or $P = P_\infty$ (under certain conditions on $\sum \lambda_i$).
			\item For each $i \in \{1, \dots, r\}$, the ramification index is $e_i = m / \gcd(m, \lambda_i)$ and the number of places in $\cF$ lying above $P_i$ is $d_i = \gcd(m, \lambda_i)$.
			\item The place $P_i$ is \textbf{totally ramified} if and only if $\gcd(m, \lambda_i) = 1$.
		\end{enumerate}
	\end{theorem}

	While recent literature \cite{Mendoza2026, 4author} focuses on the case where $d_i = 1$ for all $i$ (total ramification), our work handles the general case where $d_i > 1$. This significantly complicates the structure of the divisors but allows for a much larger class of function fields and codes.
	
	The genus $g$ of $\cF$ is calculated using the Riemann-Hurwitz formula. Let $d_\infty = \gcd(m, \sum_{i=1}^r \lambda_i)$. The genus is given by:
	\begin{equation} \label{eq:genus}
		g = \frac{(m-1)(r-1) - \sum_{i=1}^r (d_i-1) - (d_\infty-1)}{2}.
	\end{equation}
	\subsection{Maximal Function Feilds} 
	A function field \(\cF/\mathbb{F}_{q^2}\) of genus \(g\) is called \emph{maximal} if the number of its \(\mathbb{F}_{q^2}\)-rational places attains the Hasse–Weil upper bound, that is,
	\[
	N(\cF)=q^2+1+2gq.
	\]

	\subsection{Dickson polynomials}
	
	The \emph{Dickson polynomials of the first kind}
	\(
	(\varphi_d)_{d \ge 0}
	\)
	are defined recursively by
	\[
	\varphi_0(x)=2,\qquad
	\varphi_1(x)=x,
	\]
	and
	\[
	\varphi_{d+1}(x)=x\varphi_d(x)-\varphi_{d-1}(x),
	\qquad d\ge 1.
	\]
	
	They satisfy the functional identity
	\[
	\varphi_d\!\left(x+\frac{1}{x}\right)=x^d+x^{-d}.
	\]
	
	Moreover, if the ground field has characteristic \(p>0\) with
	\(
	\gcd(p,d)=1,
	\)
	then \(\varphi_d(x)\) is separable \cite{TTop}.
	\subsection{Galois Action and Invariant Divisors}
	
	The extension $\cF/\cL$ is a cyclic Galois extension of degree $m$. The Galois group $\mathcal{G} = Gal(\cF/\cL)$ is generated by the automorphism $\sigma: y \mapsto \zeta_m y$, where $\zeta_m$ is a primitive $m$-th root of unity in $\mathbb{F}_q$ (assuming $m | (q-1)$). 
	
	A divisor $A$ of $\cF$ is called \textbf{invariant} if $\sigma(A) = A$ for any $\sigma \in \mathcal{G}$. In the context of Kummer extensions, any invariant divisor can be decomposed based on the places lying above the rational places of $\cL$. Let $P \in \mathbb{P}_\cL$ and let $\{Q_1, \dots, Q_d\}$ be the places in $\cF$ lying above $P$. An invariant divisor $A$ supported on these places must take the form $A = n \sum_{j=1}^d Q_j$. 
	
	Understanding the dimension of the Riemann-Roch spaces $\mathscr{L}(A)$ for such invariant divisors is key to our strategy. By restricting these spaces to the subfield $\cL$ and using the properties of the norm and trace maps, we can characterize the non-speciality of $A$ without relying on the combinatorial complexity of multi-point Weierstrass semigroups used in \cite{Castellanos2026, Mendoza2026}.
	
	\subsection{Algebraic Geometry Codes}
	
	Let $\cF/\mathbb{F}_q$ be a function field of genus $g$. Let $\mathcal{P} = \{P_1, \dots, P_n\}$ be a set of $n$ distinct rational places of $\cF$ and define the divisor $D = P_1 + \dots + P_n$. Let $G$ be a divisor of $\cF$ such that $\operatorname{supp}(G) \cap \operatorname{supp}(D) = \emptyset$. The functional algebraic geometry (AG) code associated with $D$ and $G$, denoted by $C_\cL(D, G)$, is defined as:
	\[ C_\cL(D, G) = \{ (f(P_1), \dots, f(P_n)) : f \in \mathcal{L}(G) \}. \]
	If $\deg(G) < n$, the dimension of the code is $k = \ell(G)$. Furthermore, by the Riemann-Roch theorem, if $\deg(G) > 2g-2$, then $k = \deg(G) - g + 1$. The minimum distance $d$ of $C_\cL(D, G)$ satisfies $d \geq n - \deg(G)$.
	
	\subsection{Linear Complementary Pairs of AG Codes}
	
	We now focus on the construction of Linear Complementary Pairs (LCP) from AG codes. A pair of linear codes $(C, E)$ of length $n$ over $\mathbb{F}_q$ is called a \textbf{Linear Complementary Pair (LCP)} if their intersection is trivial and their direct sum spans the entire space, i.e., $C \cap E = \{0\}$ and $C \oplus E = \mathbb{F}_q^n$.
	
	For two AG codes $C = C_\cL(D, G)$ and $E = C_\cL(D, H)$, the LCP property is fundamentally determined by the speciality of the divisors involved. 
	For \(G,H \in \Div(\cF)\), we define
	\[
	\gcd(G,H):=\sum_{P\in\mathbb{P}_\cF}\min\{\nu_P(G),\nu_P(H)\}P,
	\]
	and
	\[
	\operatorname{lmd}(G,H):=\sum_{P\in\mathbb{P}_\cF}\max\{\nu_P(G),\nu_P(H)\}P.
	\] According to \cite[Theorema~3.5]{Bhowmick2024}, the characterization is given as follows:
	
	\begin{theorem} 
		\label{prop:LCP_AG_Criterion}
		The pair $(C_\cL(D, G), C_\cL(D, H))$ is an LCP if
		\begin{enumerate}
			\item $\ell(G) + \ell(H) = n$, and
			\item $\gcd(G,H)$ and $\operatorname{lmd}(G,H)-D$ are both non-special divisors of degree $g-1$.
		\end{enumerate}
	\end{theorem}

	In the following sections, we employ the arithmetic of Kummer extensions together with the restriction techniques introduced by Maharaj \cite{Maharaj2004} and the action of the Galois group on invariant divisors to explicitly construct the divisors \(G\) and \(H\), thereby avoiding the combinatorial machinery of generalized Weierstrass semigroups. Specifically, by ensuring that $\deg(G) + \deg(H) = n + g - 1$, the first condition is satisfied under mild assumptions, and the LCP property reduces to the identification of non-special divisors of degree $g-1$.
	
	Throughout this paper, we adopt the Iverson bracket notation:
	\[
	[P]=1 \text{ if } P \text{ is true}, \qquad [P]=0 \text{ otherwise}.
	\]
	\section{Construction of Non-Special Divisors}
	
	In this section, we present the core results of our study. We aim to provide a full characterization of non-special divisors of degree $g$ and $g-1$ in the context of general Kummer extensions. The complexity of these extensions, particularly when dealing with non-totally ramified places, requires a rigorous combinatorial and algebraic treatment.
	
	Consider the Kummer extension \(\mathcal{F}/\mathcal{L}\) defined over an algebraically closed field \(\mathcal{K}\) of characteristic \(p>0\), where \(\gcd(p,m)=1\). Let \(\mathcal{L}=\mathcal{K}(x)\) be the rational function field and let \(\mathcal{F}=\mathcal{K}(x,y)\) be given by
	\begin{equation}\label{eq:Kummer_detailed}
		y^{m}=\prod_{i=1}^{r}(x-\alpha_i)^{\lambda_i}.
	\end{equation}
	Here, \(m,r\in\mathbb{Z}_{>0}\), the integers \(\lambda_i>0\) satisfy
	\(
	\gcd(m,\lambda_1,\dots,\lambda_r)=1,
	\)
	and \(\alpha_1,\dots,\alpha_r\in\mathcal{K}\) are pairwise distinct. Under these assumptions, \(\mathcal{F}/\mathcal{L}\) is a cyclic Kummer extension of degree \(m\). The ramification behavior of the extension is determined by the exponents \(\lambda_i\) and the branch places \(\alpha_i\).
	
	According to the theory of Kummer extensions \cite{Stichtenoth2009}, the ramification behavior of a rational place $Q_i \in \mathbb{P}(\cL)$ associated with the zero of $x-\alpha_i$ is determined by the exponent $\lambda_i$. Specifically, the divisor of $x-\alpha_i$ in $\cF$ is given by
	\begin{equation} \label{eq:div_xi}
		\operatorname{div}(x-\alpha_i) = \frac{m}{d_i} \sum_{j=1}^{d_i} P_{i,j} - \operatorname{div}_{\infty}(x),
	\end{equation}
	where $d_i = \gcd(m, \lambda_i)$ and $\{P_{i,1}, \dots, P_{i,d_i}\}$ are the distinct places of $\cF$ lying above $Q_i$. Each such place $P_{i,j}$ has a ramification index of $e(P_{i,j}|Q_i) = m/d_i$. In the particular case where $\gcd(m, \lambda_i)=1$, the place $Q_i$ is \textit{totally ramified}, and we denote the unique place above it simply by $P_i$.
	
	Furthermore, let $\Lambda \coloneqq \sum_{i=1}^{r} \lambda_i$. The decomposition of the place at infinity $Q_{\infty} \in \mathbb{P}(\cL)$ is similarly governed by $\gcd(m, \Lambda)$. Specifically, $Q_{\infty}$ decomposes into $d_{\infty} = \gcd(m, \Lambda)$ distinct places in $\cF$, each having a ramification index of
	\begin{equation}
		e(P_{\infty,j}|Q_{\infty}) = \frac{m}{\gcd(m, \Lambda)}.
	\end{equation}
	
	\subsection{Invariant Divisors and the Restriction Map}
	
	Let $\cF/\cL$ be a Kummer extension with Galois group $\mathcal{G} = \operatorname{Aut}(\cF/\cL) = \langle \sigma \rangle$ of order $m$. A divisor $A \in \operatorname{Div}(\cF)$ is said to be \textit{invariant} under the action of $\mathcal{G}$ if $\sigma(A) = A$. To analyze the Riemann-Roch spaces of such divisors, we utilize the restriction operator introduced by Maharaj \cite{Maharaj2004}.
	
	\begin{definition} \label{def:restriction_map}
		For any divisor $D \in \operatorname{Div}(\cF)$, the restriction operator $R: \operatorname{Div}(\cF) \to \operatorname{Div}(\cL)$ is defined as:
		\begin{equation}
			R(D) := \sum_{Q \in \mathbb{P}(\cL)} \min_{P|Q} \left\{ \left\lfloor \frac{v_P(D)}{e(P|Q)} \right\rfloor \right\} \cdot Q,
		\end{equation}
		where $P|Q$ denotes the places in $\cF$ lying over the place $Q$ in $\cL$, and $e(P|Q)$ is the respective ramification index.
	\end{definition}
	
	The following lemma provides a fundamental decomposition of the dimension of the Riemann-Roch space $\mathcal{L}_\cF(A)$ into a sum of dimensions of spaces defined over the subfield $\cL$.
	
	\begin{lem} \label{lem:maharaj}
		Let $A \in \operatorname{Div}(\cF)$ be an invariant divisor. Then, the dimension $\ell_\cF(A)$ can be computed as:
		\begin{equation}
			\ell_{\cF}(A) = \sum_{t=0}^{m-1} \ell_{\cL}\big(R(A + \operatorname{div}(y^t))\big).
		\end{equation}
	\end{lem}
	
	\begin{remark}
		\rm{Following Maharaj’s derivation \cite{Maharaj2004}, for any invariant divisor $A$ of $\mathcal{F}$, the condition $\sum_{j=0}^{m-1} a_j(x) y^j \in \mathcal{L}(A)$ implies that each individual term $a_j(x) y^j \in \mathcal{L}(A)$, where the divisor associated with each term is also invariant. Consequently, any invariant divisor $A$ with $\dim \mathcal{L}(A) > 0$ is linearly equivalent to an effective invariant divisor. Thus, the classification of invariant non-special divisors of degree $g$ can be restricted to effective divisors without loss of generality. Note that this equivalence does not necessarily extend to non-special divisors of degree $g-1$.}
	\end{remark}
	
	\subsection{The Main Characterization Theorem}
	
	We now state and prove the primary theorem that characterizes non-speciality. Let $A$ be an invariant effective divisor of the form:
	\[ A \coloneqq n_{0} \frac{\operatorname{div}_{\infty}(x)}{m/\operatorname{gcd}(m, \Lambda)} + \sum_{i=1}^{r} n_{i} \frac{\operatorname{div}_{0}(x-\alpha_{i})}{m/\operatorname{gcd}(m, \lambda_{i})}. \]
	For simplicity in notation, we define $B(n_0, j)$ as:
	\begin{equation}
		\label{eq:bound_B}
		B(n_{0},j) \coloneqq -1 + \left\lceil \frac{\sum_{i=1}^{r}(-j \lambda_{i}) \mod m - n_{0} \operatorname{gcd}(m, \Lambda)}{m} \right\rceil.
	\end{equation}
	
	\begin{theorem} \label{thm:class_n_s_inv_eff_div}
		Let $\cF=\cK(x,y)$ be the algebraic function field over $\cK$ defined by
		\begin{equation}
			\label{eq:Kummer}
			y^{m}=\prod_{i=1}^{r}(x-\alpha_i)^{\lambda_i},
		\end{equation}
		where $m,r\in\mathbb{Z}_{>0}$, $\cK$ is an algebraically closed field of characteristic $p>0$ satisfying $(p,m)=1$, $\lambda_1,\ldots,\lambda_r\in\mathbb{Z}$ satisfy
		\(
		\gcd(m,\lambda_1,\ldots,\lambda_r)=1,
		\)
		and $\alpha_1,\ldots,\alpha_r$ are pairwise distinct elements of $\cK$.
		
		Let $\Lambda \coloneqq \sum_{i=1}^{r} \lambda_{i}$ and let 
		\begin{equation*}
			A \coloneqq n_{0} \frac{\operatorname{div}_{\infty}(x)}{m/\gcd(m, \Lambda)} + \sum_{i=1}^{r} n_{i} \frac{\operatorname{div}_{0}(x - \alpha_{i})}{m/\gcd(m, \lambda_{i})},
		\end{equation*}
		with $n_{0}, \dots, n_{r} \in \mathbb{N}$, be an $\operatorname{Aut}(\cK(x,y)/\cK(x))$-invariant effective divisor. The following statements are equivalent:
		\begin{enumerate}
			\item $\deg A = g$ and $\ell(A) = 1$.
			\item $\deg A = g$, $n_{0} < \frac{m}{\gcd(m, \Lambda)}$, and $n_{i} < \frac{m}{\gcd(m, \lambda_{i})}$ for each $i \in \{1, \dots, r\}$, and for each $j \in \{1, \dots, m-1\}$:
			\begin{equation*}
				\left| \{ i \in \{1, \dots, r\} \mid n_{i} \gcd(m, \lambda_{i}) \ge (j \lambda_{i}) \bmod m > 0 \} \right| \le B(n_{0}, j).
			\end{equation*}
			\item $n_{0} < \frac{m}{\gcd(m, \Lambda)}$, and $n_{i} < \frac{m}{\gcd(m, \lambda_{i})}$ for each $i \in \{1, \dots, r\}$, and for each $j \in \{1, \dots, m-1\}$:
			\begin{equation*}
				\left| \{ i \in \{1, \dots, r\} \mid n_{i} \gcd(m, \lambda_{i}) \ge (j \lambda_{i}) \bmod m > 0 \} \right| = B(n_{0}, j).
			\end{equation*}
		\end{enumerate}
	\end{theorem}
	
	
	To prove Theorem~\ref{thm:class_n_s_inv_eff_div}, we first establish a fundamental summation identity. The following lemma provides the combinatorial link between the local conditions on the coefficients and the global genus of the function field.
	
	\begin{lem} \label{lem:sum_B}
		For any $n_0 < \frac{m}{\gcd(m, \Lambda)}$, the sum of the bounds $B(n_0, j)$ satisfies:
		\[
		\sum_{j=1}^{m-1} B(n_0, j) = g - n_0 \gcd(m, \Lambda).
		\]
	\end{lem}
	
	\begin{proof}
		We begin by expanding the definition of $B(n_0, j)$ and substituting the modular terms:
		\begin{align*}
			\sum_{j=1}^{m-1} B(n_0, j) =  \sum_{j=1}^{m-1} \left( \sum_{i=1}^{r} \left\lceil \frac{j \lambda_i}{m} \right\rceil + \left\lceil \frac{-j \Lambda - n_0 \gcd(m, \Lambda)}{m} \right\rceil - 1 \right).
		\end{align*}
		
		To evaluate the sum, we analyze the two primary subexpressions separately. First, consider the term involving $\Lambda$; by analysing the possible remainders modulo $m/\gcd(m, \Lambda)$, we obtain:
		\begin{align*}
			&\mathrel{\phantom{=}} 2 \sum_{j=1}^{m-1} \left\lceil \frac{-j \Lambda - n_0 \gcd(m, \Lambda)}{m} \right\rceil =(m-1)(1 - \Lambda) - \gcd(m, \Lambda)(1 + 2n_0) + 1.
		\end{align*}
		
		Next, for each $i \in \{1, \dots, r\}$, we evaluate the sum of the ceiling terms:
		\begin{align*}
			2 \sum_{j=1}^{m-1} \left\lceil \frac{j \lambda_i}{m} \right\rceil = (m-1)(1 + \lambda_i) - \gcd(m, \lambda_i) + 1.
		\end{align*}
		
		Finally, combining these results into the main summation:
		\begin{align*}
			\sum_{j=1}^{m-1} B(n_0, j) = \frac{1}{2} \left( \sum_{i=1}^r (m - \gcd(m, \lambda_i)) + (m - \gcd(m, \Lambda)) - 2(m-1) \right) - n_0 \gcd(m, \Lambda).
		\end{align*}
		Recalling the genus formula $g = 1 - m + \frac{1}{2} \left( \sum_{i=1}^r (m - \gcd(m, \lambda_i)) + (m - \gcd(m, \Lambda)) \right)$, we conclude:
		\begin{equation*}
			\sum_{j=1}^{m-1} B(n_0, j) = g - n_0 \gcd(m, \Lambda).
		\end{equation*}
	\end{proof}
	
	\begin{proof}[Proof of Theorem \ref{thm:class_n_s_inv_eff_div}]
		We utilize the decomposition of the Riemann-Roch space $\mathscr{L}(A)$ via the restriction map $R(\cdot)$ to the divisors of the rational subfield $\cK(x)$. 
		
		\noindent\textbf{Equivalence (1 $\iff$ 2):} \\
		By Lemma \ref{lem:maharaj}, $\ell(A)=1$ for an effective invariant divisor $A$ if and only if $\ell_{\cK(x)}(R(A)) = 1$ and $\deg(R(A + \divv(y^t))) < 0$ for all $t \in \{1, \dots, m-1\}$. 
		The first condition is equivalent to the the bounds $n_0 < \frac{m}{\gcd(m, \Lambda)}$ and $n_i < \frac{m}{\gcd(m, \lambda_i)}$. For the second condition, we apply the change of variables $j = m-t$ and compute:
		\begin{align*}
			\deg(R(A + \divv(y^t))) &= \left\lfloor \frac{\deg A - \sum_{i=1}^{r} (n_i \gcd(m, \lambda_i) + t\lambda_i) \bmod m}{m} \right\rfloor \\
			&= \left\lfloor \frac{n_0 \gcd(m, \Lambda) - \sum_{i=1}^r (-j\lambda_i) \bmod m + m\left|C(n_0, j)\right|}{m} \right\rfloor,
		\end{align*}
		where $C(n_0,j)$ is the set $\{ i \in \{1, \dots, r\} \mid n_{i} \gcd(m, \lambda_{i}) \ge (j \lambda_{i}) \bmod m > 0 \}$.\footnote{Informally, we are looking at which values of $n_i \operatorname{gcd}(m,\lambda_i)$, when summed with $(-j \lambda_i) \bmod m$, “overflow”.}
		The condition that this value is strictly less than zero is equivalent to the integer inequality $|C(n_0, j)| \leq B(n_0, j)$, where $B(n_0, j)$ is defined as in \eqref{eq:bound_B}.
		
		\noindent\textbf{Equivalence (2 $\iff$ 3):} \\
		To connect the degree of $A$ to the genus $g$, we observe that the total degree of an invariant divisor (excluding the $n_0$ term) can be written as the sum of the sizes of the sets $C(n_0, j)$:
		\begin{align*}
			\deg A-n_{0}\gcd(m,\Lambda)
			&=\sum_{i=1}^{r} n_i \gcd(m,\lambda_i) \\
			&=\sum_{i=1}^{r}\sum_{j=1}^{m-1}
			\left[n_i\gcd(m,\lambda_i)\ge (j\lambda_i)\bmod m>0\right] =\sum_{j=1}^{m-1}|C(n_0,j)|.
		\end{align*}
		When $\left|C(n_0, j)\right| \leq B(n_0, j)$ for each $j\in\left\{1,\ldots,m-1\right\}$, we have, by summing over all $j$ and applying Lemma \ref{lem:sum_B},
		\begin{equation*}
			\sum_{j=1}^{m-1} |C(n_0, j)| \leq \sum_{j=1}^{m-1} B(n_0, j) = g - n_0 \gcd(m, \Lambda),
		\end{equation*}
		and so the equality $\deg A = g$ holds if and only if each equality $\left|C(n_0, j)\right| = B(n_0, j)$, for each $j\in\left\{1,\ldots,m-1\right\}$, holds.
	\end{proof}
	
	\subsection{Constructive Examples and Comparison with Previous Results}
	
	The following examples illustrate the flexibility of Theorem~\ref{thm:class_n_s_inv_eff_div} in the construction of non-special divisors, as well as its effectiveness in recovering and extending several recent results in the literature. Furthermore, the proof of Theorem~\ref{thm:class_n_s_inv_eff_div} relies directly on the decomposition machinery introduced by Maharaj, avoiding the intermediate framework of minimal elements of multi-Weierstrass semigroups.
	
	\begin{example}
		\rm{
			Assume that
			\(
			\lambda_1=\cdots=\lambda_r=1.
			\) Then, for each $j\in\{1,\ldots,m-1\}$, the quantity $B(0,j)$ takes the form
			\[
			B(0,j)
			=
			-1+\left\lceil \frac{r(m-j)}{m}\right\rceil
			=
			r-1-\left\lfloor \frac{rj}{m}\right\rfloor.
			\]

			Assume \(0\le n_1\le \cdots \le n_r<m\). By Theorem~\ref{thm:class_n_s_inv_eff_div}, the tuple \((0,n_1,\ldots,n_r)\) defines a non-special divisor of degree \(g\) if and only if
			\[
			\#\{\,i:n_i=j\,\}=B(0,j)-[j<m-1]B(0,j+1)
			\]
			for all \(j\in\{1,\ldots,m-1\}\). 
			Consequently, the integers $n_i$ are uniquely determined and satisfy
			\[
			n_i
			=
			\max\!\left\{
			0,\
			\left\lceil \frac{m(i-1)}{r}\right\rceil-1
			\right\}.
			\]
			
			In particular, if \(\gcd(m,r)=1\), then
			\(
			n_i=\left\lfloor \frac{m(i-1)}{r}\right\rfloor,
			\)
			recovering the explicit formulas in~\cite[Theorem~8]{Moreno2024}. In the hyperelliptic case \((m=2)\), any sum of \(g\) distinct zeros of \(y\) is non-special.
			
			More generally, if
			\(
			\lambda_1=\cdots=\lambda_r=\lambda
			\quad\text{with}\quad
			\gcd(m,\lambda)=1,
			\)
			then an analogous description follows, since the map
			\(
			j\longmapsto (j\lambda)\bmod m
			\)
			is a permutation of $\{1,\ldots,m-1\}$.
		}
	\end{example}
	
	\begin{example}\rm{
			We revisit the criteria introduced in~\cite{4author} for constructing divisors associated with LCP codes. Let \((n_0,\ldots,n_r)\in\mathbb{N}^{r+1}\) satisfy:
			\begin{itemize}
				\item \(n_0=0\);
				\item \(n_i=0\) whenever \(\gcd(m,\lambda_i)\neq 1\);
				\item \(n_i\in\{0,m-1\}\) whenever \(\lambda_i\neq 1\);
				\item \(0\le n_i\le m-1\) whenever \(\lambda_i=1\).
			\end{itemize}
			
			Then Theorem~\ref{thm:class_n_s_inv_eff_div} implies that the existence of such divisors is equivalent to
			\(
			0\le B(0,m-1)\le \cdots \le B(0,1),
			\)
			together with
			\(
			B(0,1)-B(0,m-1)
			\le
			\left|\{\,i\in\{1,\ldots,r\}:\lambda_i=1\,\}\right|,
			\)
			and
			\(
			B(0,1)
			\le
			\left|\{\,i\in\{1,\ldots,r\}:\gcd(m,\lambda_i)=1\,\}\right|.
			\)
			
			Moreover, \(B(0,j)\) coincides with the quantity \(-1+S_j\) in~\cite{4author}, so the auxiliary set \(V_F\) introduced there becomes unnecessary.
			
			A further extension is obtained by allowing
			\(
			n_i\in\left\{0,\frac{m}{\gcd(m,\lambda_i)}-1\right\}
			\)
			whenever \(\lambda_i\neq 1\).
		}
	\end{example}
	From the two examples above, it follows that similar constructions can be obtained by partitioning $\{1,\ldots,r\}$ into two subsets $I$ and $J$, where the coefficients $n_i$ are prescribed for all $i\in J\cup\{0\}$ and the condition $\lambda_i=1$ holds for every $i\in I$. In this setting, the existence of invariant non-special divisors is equivalent to the validity of a corresponding system of inequalities. Moreover, any two such solutions coincide up to a permutation of the coefficients associated with the indices in $I$. Theorem~\ref{thm:class_n_s_inv_eff_div} also provides an effective procedure to determine all effective invariant non-special divisors of degree \(g\) through a finite combinatorial search once the parameters \(m\) and \(\lambda_1,\ldots,\lambda_r\) are fixed.
	
	\begin{example}\label{Ex3.7}
		\rm{
			Consider the Kummer extension \(\cF/\mathbb{F}_q(x)\) defined by
			\[
			y^{6}
			=
			(x-\alpha_{1})(x-\alpha_{2})(x-\alpha_{3})
			(x-\beta)^{5}
			(x-\gamma)^{3}.
			\]
			In this case,
			\(
			m=6,
			\;
			(\lambda_1,\ldots,\lambda_5)=(1,1,1,3,5).
			\)
			
			Up to permutations among the coefficients corresponding to the indices satisfying \(\lambda_i=1\), namely \(n_1,n_2,n_3\), the following tuples
			\(
			(n_0,n_1,\ldots,n_5)
			\)
			determine all effective invariant non-special divisors of degree \(g=9\):
			
			\begin{center}
				\small
				\renewcommand{\arraystretch}{1.2}
				\begin{tabular}{|c|c|c|}
					\hline
					$(0,0,1,3,0,5)$ & $(1,0,1,3,0,4)$ & $(3,0,1,1,0,4)$ \\ \hline
					$(0,0,2,4,1,0)$ & $(1,0,2,3,0,3)$ & $(3,0,1,2,0,3)$ \\ \hline
					$(0,1,1,3,0,4)$ & $(1,0,3,3,0,2)$ & $(3,0,1,3,0,2)$ \\ \hline
					$(0,1,2,3,0,3)$ & $(1,0,3,4,0,1)$ & $(3,0,1,4,0,1)$ \\ \hline
					$(0,1,3,3,0,2)$ & $(1,0,3,5,0,0)$ & $(3,0,1,5,0,0)$ \\ \hline
					$(0,1,3,4,0,1)$ & $(2,0,0,4,1,0)$ & $(4,0,0,2,1,0)$ \\ \hline
					$(0,1,3,5,0,0)$ & $(2,0,1,3,0,3)$ & $(4,0,1,3,0,1)$ \\ \hline
					$(1,0,0,3,0,5)$ & $(3,0,0,1,0,5)$ & $(5,0,1,3,0,0)$ \\ \hline
				\end{tabular}
			\end{center}
			
			This example corresponds to the function field  considered in~\cite[Example~10]{4author}. However, the method developed there yields only the two effective invariant non-special divisors associated with the tuples
			\(
			(0,1,3,0,5,0) \) and
			\(
			(0,1,3,5,0,0),
			\) 
			whereas Theorem~\ref{thm:class_n_s_inv_eff_div} provides the complete classification.
			
			On the other hand, if \(m=17\), \(r=2\), and \((\lambda_1,\lambda_2)=(1,2)\), then the theorem immediately implies that no effective invariant non-special divisor of degree \(g\) exists.
		}
	\end{example}
	\subsection{Existence Conditions and Explicit Coefficients for $\lambda_i \in \{1, m/2\}$}

	In this section, we provide a detailed characterization of the coefficients for effective non-special divisors in Kummer extensions where a subset of the ramified places has multiplicity $m/2$. We analyze the cases where one or two such places exist, providing explicit formulas for $n_i$.
	
	\subsubsection{Existence Conditions for \(\lambda_i\in\{1,m/2\}\)}
	
	Assume that \(m\) is even and let \(s\in\{1,\ldots,r-1\}\). Suppose
	\(
	\lambda_1=\cdots=\lambda_s=1
	\)
	and
	\(
	\lambda_{s+1}=\cdots=\lambda_r=\frac{m}{2}.
	\)
	
	For \(N\in\{0,\ldots,r-s\}\), let \(\mathcal P(N)\) denote the existence of an effective divisor
	\[
	A=\sum_{i=1}^s n_iP_i+\sum_{i=1}^{r-s}n_{s+i}\frac{\operatorname{div}_0(x-\alpha_{s+i})}{2}
	\]
	satisfying \(\deg(A)=g\), \(\ell(A)=1\), and
	\(
	n_{s+1}=\cdots=n_{s+N}=1
	\)
	together with
	\(
	n_{s+N+1}=\cdots=n_r=0.
	\) 
	Define
	\[
	\mathcal B(N,j)
	:=
	B(0,j)
	-
	\left|
	\left\{
	i\in\{1,\ldots,r-s\}:
	[i\le N]\frac m2
	\ge
	\left(j\frac m2\right)\bmod m>0
	\right\}
	\right|
	=
	B(0,j)-(j\bmod 2)N.
	\]
	
	Then \(\mathcal P(N)\) holds if and only if
	\(
	0\le \mathcal B(N,m-1)\le\cdots\le \mathcal B(N,1)\le s.
	\)
	
	Moreover,
	\[
	\mathcal B(N,j)
	=
	-1+
	\left\lceil
	\frac{
		s(m-j)+(r-s)\frac m2(j\bmod 2)
	}{m}
	\right\rceil
	-(j\bmod 2)N,
	\]
	and for \(j\in\{1,\ldots,m-2\}\), we have
	\[
	\mathcal B(N,j)-\mathcal B(N,j+1)
	=
	\left\lfloor
	\frac{
		s(j+1)-(r-s)\frac m2((j+1)\bmod 2)
	}{m}
	\right\rfloor
	-
	\left\lfloor
	\frac{
		sj-(r-s)\frac m2(j\bmod 2)
	}{m}
	\right\rfloor
	+(-1)^jN.
	\]
	\begin{prop}\label{prop:existence 1 of m/2} 
		Assume that \(r-s=1\). Then:
		\begin{enumerate}
			\item The property \(\mathcal{P}(0)\) holds whenever
			\(
			r\ge \frac{m}{2}-\left(\frac{m}{2}\bmod 2\right).
			\)
			
			\item The property \(\mathcal{P}(1)\) holds whenever
			\(
			r\ge \frac{m}{2}+2.
			\)
		\end{enumerate}
	\end{prop}
	\begin{proof}
		We verify the boundary conditions
		\(0\le \mathcal B(N,m-1)\) and \(\mathcal B(N,1)\le s\), together with the monotonicity condition
		\(
		\mathcal B(N,j)\ge \mathcal B(N,j+1).
		\)
		
		\begin{enumerate}
			\item Assume \(N=0\) and
			\(
			r\ge \frac m2-\left(\frac m2\bmod2\right).
			\)
			Then
			\[
			\mathcal B(0,m-1)
			=
			-1+
			\left\lceil\frac{(r-1)+\frac m2}{m}\right\rceil
			\ge0,
			\]
			and
			\[
			\mathcal B(0,1)
			=
			-1+
			\left\lceil\frac{(r-1)(m-1)+\frac m2}{m}\right\rceil
			\le r-1=s.
			\]
			
			Moreover, if \(j\) is odd, then
			\[
			\mathcal B(0,j)-\mathcal B(0,j+1)
			=
			\left\lfloor\frac{(r-1)(j+1)}{m}\right\rfloor
			-
			\left\lfloor\frac{(r-1)j-\frac m2}{m}\right\rfloor
			\ge0.
			\]
			
			If \(j\) is even, then
			\[
			\mathcal B(0,j)-\mathcal B(0,j+1)
			=
			\left\lfloor
			\frac{(r-1)j+\left(r-1-\frac m2\right)}{m}
			\right\rfloor
			-
			\left\lfloor\frac{(r-1)j}{m}\right\rfloor.
			\]
			
			This is nonnegative if \(r-1-\frac m2\ge0\). Otherwise,
			\(
			r-1\in\left\{\frac m2-1,\frac m2-1-\left(\frac m2\bmod2\right)\right\},
			\)
			so \(\gcd(r-1,m)=1\), whence
			\[
			((r-1)j)\bmod m
			=
			2\left((r-1)\frac j2\right)\bmod \frac m2
			\ge2,
			\]
			which again yields nonnegativity.
			
			\item Assume \(N=1\) and \(r\ge \frac m2+2\). Then
			\[
			\mathcal B(1,m-1)
			=
			-1+
			\left\lceil\frac{(r-1)+\frac m2}{m}\right\rceil
			-1
			\ge0,
			\]
			and
			\[
			\mathcal B(1,1)
			=
			-1+
			\left\lceil\frac{(r-1)(m-1)+\frac m2}{m}\right\rceil
			-1
			\le r-1=s.
			\]
			
			Further,
			\[
			\mathcal B(1,j)-\mathcal B(1,j+1)
			=
			\left\lfloor
			\frac{
				(r-1)(j+1)-\frac m2((j+1)\bmod2)
			}{m}
			\right\rfloor
			-
			\left\lfloor
			\frac{
				(r-1)j-\frac m2(j\bmod2)
			}{m}
			\right\rfloor
			+(-1)^j.
			\]
			This is nonnegative since
			\(
			(r-1)+\frac m2\ge m+1
			\)
			for odd \(j\), and
			\(
			(r-1)-\frac m2\ge1
			\)
			for even \(j\).
		\end{enumerate}
	\end{proof}
	
	Using the same argument as in the previous proposition, one obtains the following existence criteria in the case \(r-s=2\).
	
	\begin{prop}\label{prop:existence 2 of m/2} 
		Assume that \(r-s=2\). Then:
		\begin{enumerate}
			\item \(\mathcal P(0)\) holds whenever
			\(
			s=r-2\ge m-1-\left(\frac m2\bmod2\right).
			\)
			
			\item \(\mathcal P(1)\) always holds, equivalently,
			\(
			s=r-2\ge1.
			\)
			
			\item \(\mathcal P(2)\) holds whenever
			\(
			s=r-2\ge m+1.
			\)
		\end{enumerate}
	\end{prop}
	
	\subsubsection{Explicit Coefficients for \(\lambda_i\in\{1,m/2\}\)}
	
	We now determine explicitly the coefficients of the divisors obtained in Proposition \ref{prop:existence 1 of m/2} and \ref{prop:existence 2 of m/2}.
	
	Assume
	$\lambda_{1}=\cdots=\lambda_{s}=1$ and $\lambda_{s+1}=\cdots=\lambda_{r}=\frac{m}{2}$.
	We consider divisors of the form
	\[
	A=
	\sum_{i=1}^{s}
	n_i P_i
	+
	\sum_{i=1}^{r-s}
	n_{s+i} \frac{\divv_0(x-\alpha_{s+i})}{2},
	\]
	where $P_i$, $1\le i\le s$, is the unique zero of $x-\alpha_i$,
	such that $A$ is effective non-special of degree $g$.
	Remember that, necessarily, $n_{s+1},\ldots,n_r\in\{0,1\}$,
	and let us denote $N=n_{s+1}+\cdots+n_r$.
	
	Assuming $n_1\le\cdots\le n_s$, the coefficients satisfy
	\begin{align*}
		n_i
		& = \max\left( \{0\} \cup \{j\in\{1,\ldots,m-1\}:
		j \le n_i \le \cdots \le n_s
		\}\right)
		\\
		& =
		\max\left(
		\{0\}\cup
		\{
		j\in\{1,\ldots,m-1\}:
		\mathcal{B}(N,j) \ge s+1-i
		\}
		\right),
		\quad
		1\le i\le s.
	\end{align*}
	Moreover,
	\[
	\mathcal{B}(N,j)=
	-1+r
	-\left\lfloor\frac{sj + \epsilon_j \frac{m}{2}(r-s)}{m}\right\rfloor
	-
	\begin{cases}
		(r-s), & \epsilon_j=0
		\\
		N, & \epsilon_j=1
	\end{cases}
	\]
	where $\epsilon_j = j \bmod 2$.
	
	Write $j=2k+\epsilon_j$. We have that
	$\mathcal{B}(N,j) \ge s+1-i$ if and only if
	\[
	k \le -1
	+ \left\lceil\frac{
		m(i-1+\epsilon_j (r-s-N)) - \epsilon_j(r-s)\frac{m}{2} - \epsilon_j s
	}{2s}\right\rceil.
	\]
	Thus, considering the cases $\epsilon_j=0$ and $\epsilon_j=1$ separately, we have
	\[
	n_i=
	\max\left\{
	0
	,
	2\left\lceil \frac{m(i-1)}{2s} \right\rceil - 2
	,
	2\left\lceil \frac{m(i-1+r-s-N) - (r-s)\frac{m}{2} - s}{2s} \right\rceil - 1
	\right\}.
	\]
	(Note that the right side is always less than $m$.)
	
	The following theorems will consider the cases $r-s=1$ and $r-s=2$
	and will give simplified formulas for certain values of $r$.
	
	\begin{theorem}
		\label{thm:explicit coeff 1 of m/2}
		
		Let $m\ge 2$ be even
		and consider the Kummer extension $\mathcal{F}/\mathcal{K}(x)$
		given by
		\[
		y^m = \prod_{i=1}^{r} (x-\alpha_i)^{\lambda_i},
		\]
		where $\gcd\left(\operatorname{char}(\mathcal{K}),m)\right)=1$,
		the elements $\alpha_1,\ldots,\alpha_r\in\mathcal{K}$ are distinct,
		\[
		\lambda_1=\cdots=\lambda_{r-1}=1,
		\qquad \lambda_r=\frac{m}{2}.
		\]
		
		Let $N\in\{0,1\}$. Then, whenever
		\[
		r \ge
		\begin{cases}
			\frac{m}{2} - \left( \frac{m}{2} \bmod 2 \right)
			& N = 0
			\\
			\frac{m}{2} + 2
			& N = 1,
		\end{cases}
		\]
		the divisor of $\mathcal{F}$
		\[
		A =
		\sum_{i=1}^{r-1} n_i P_i
		+
		[N=1] \frac{\divv_0(x-\alpha_r)}{2}
		\]
		is non-special of degree $g=g(\mathcal{F})$,
		where the coefficients $n_1,\ldots,n_{r-1}$ are given by
		\[
		n_i = 
		\max\left\{
		0
		,
		2\left\lceil \frac{m(i-1)}{2(r-1)} \right\rceil - 2
		,
		2\left\lceil \frac{m(i-N-\frac12) - (r-1)}{2(r-1)} \right\rceil - 1
		\right\}.
		\]
		
		In particular, we have the following special cases.
		\begin{enumerate}
			\item
			If $N=0$ and
			$\frac{m}{2} - \left( \frac{m}{2} \bmod 2 \right) \le r \le \frac{m}{2}$,
			the divisor $A$ is $\sum_{i=1}^{r-1} n_i P_i$, with
			\[
			n_i =
			2\left\lceil \frac{m(i-\frac12) - (r-1)}{2(r-1)} \right\rceil - 1,
			\qquad
			1 \le i \le r-1.
			\]
			
			\item
			If $N=1$ and
			$r=\frac{m}{2} + 2$,
			then we have
			\[
			n_{2+i} = 2i, \qquad 1\le i\le \frac{m}{2}-1;
			\qquad n_1=n_2=0.
			\]
		\end{enumerate}
	\end{theorem}
	\begin{proof}
		he general expression for the coefficients $n_i$ follows from the computations preceding this theorem together with Proposition \ref{prop:existence 1 of m/2}. It remains to show the special cases.
		
		Suppose $N=0$ and $r \le \frac{m}{2}$. Define
		\[
		I_0(i)=\frac{m(i-1)}{2(r-1)}
		\qquad
		\mathrm{and}
		\qquad
		I_1(i) = \frac{m(i-\frac12) - (r-1)}{2(r-1)}.
		\]
		Then
		\[
		I_1(i) - I_0(i) = \frac{m/2 - (r-1)}{2(r-1)} = I_1(1) > 0.
		\]
		Consequently,
		$2\left\lceil I_1(i) \right\rceil - 1
		>
		2\left\lceil I_0(i) \right\rceil - 2$, and therefore
		\[
		n_i = 2\left\lceil I_1(i) \right\rceil - 1.
		\]
		
		The case $N=1$ and $r=\frac{m}{2}+2$ follows by analogous arguments.
	\end{proof}
	
	The following theorem has a similar proof.
	
	\begin{theorem}
		\label{thm:explicit coeff 2 of m/2}
		
		Let $m\ge 2$ be even
		and consider the Kummer extension $\mathcal{F}/\mathcal{K}(x)$
		given by
		\[
		y^m = \prod_{i=1}^{r} (x-\alpha_i)^{\lambda_i},
		\]
		where $\gcd\left(\operatorname{char}(\mathcal{K}),m)\right)=1$,
		the elements $\alpha_1,\ldots,\alpha_r\in\mathcal{K}$ are distinct,
		\[
		\lambda_1=\cdots=\lambda_{r-2}=1,
		\qquad
		\lambda_{r-1}=\lambda_r=\frac{m}{2}.
		\]
		
		Let $N\in\{0,1,2\}$. Then, whenever
		\[
		r \ge
		\begin{cases}
			m+1 - \left( \frac{m}{2} \bmod 2 \right), & N=0
			\\
			3, & N=1
			\\
			m+3, & N=2,
		\end{cases}
		\]
		the divisor of $\mathcal{F}$
		\[
		A = \sum_{i=1}^{r-2} n_i P_i
		+
		[N \ge 1] \frac{\divv_0(x-\alpha_{r-1})}{2}
		+
		[N \ge 2] \frac{\divv_0(x-\alpha_r)}{2}
		\]
		is non-special of degree $g = g(\mathcal{F})$,
		where the coefficients $n_1,\ldots,n_{r-2}$ are given by
		\[
		n_i =
		\max\left\{
		0
		,
		2 \left\lceil \frac{m(i-1)}{2(r-2)} \right\rceil - 2
		,
		2 \left\lceil \frac{m(i-N) - (r-2)}{2(r-2)} \right\rceil - 1
		\right\}.
		\]
		
		In particular, we have the following cases:
		
		\begin{enumerate}
			\item If $N=0$, we have $A = \sum_{i=1}^{r-2} n_i P_i$.
			\item
			If $N=1$ and $r=\frac{m}{2}+2$, we have
			\[
			n_{i+1} = 2i-1,\qquad 1\le i\le \frac{m}{2}-1;
			\qquad n_0=0.
			\]
			\item
			If $N=2$ and $r=m+3$, we have
			\[
			n_{2i+2} = n_{2i+3} = 2i,\qquad 1\le i\le \frac{m}{2}-1;
			\qquad n_1=n_2=n_3=0.
			\]
		\end{enumerate}
		
	\end{theorem}

	\subsection{Existence Conditions and Explicit Coefficients for $\lambda_i \in \{1, 2\}$}
	\subsubsection{Existence Conditions for \(\lambda_i\in\{1,2\}\)}
	
	Assume throughout that \(m\) and \(r-1\) are even,
	\(
	\lambda_1=\cdots=\lambda_{r-1}=1,
	\)
	and
	\(
	\lambda_r=2.
	\)
	Set
	\(
	\Lambda=\sum_{i=1}^r\lambda_i=r+1,
	\)
	so that \(\Lambda\) is even.
	
	We consider effective invariant divisors of degree \(g\) of the form
	\begin{equation}
		\label{eq:A n_i, 1 of 2}
		A=
		n_0\frac{\operatorname{div}_{\infty}(x)}
		{m/\gcd(m,\Lambda)}
		+
		\sum_{i=1}^{r-1}n_iP_i
		+
		n_r\frac{\operatorname{div}_{0}(x-\alpha_r)}{m/2},
	\end{equation}
	where \(P_i\) is the unique zero of \(x-\alpha_i\). Let \(\mathcal Q(n_0,n_r)\) denote the condition that there is a divisor $A$ as in (\ref{eq:A n_i, 1 of 2}), for some choice of coefficients $n_1,\ldots,n_{r-1}$, such that \(\deg(A)=g\) and \(\ell(A)=1\).
	
	Define
	\[
	R(n_0,n_r,j)
	=
	B(n_0,j)
	-
	\left[
	2n_r\ge (2j)\bmod m>0
	\right].
	\]
	Then \(\mathcal Q(n_0,n_r)\) is equivalent to
	\(
	n_0<\frac{m}{\gcd(m,\Lambda)},
	\)
	\(
	n_r<\frac m2,
	\)
	and the monotonicity of the sequence
	\[
	0 \le R(n_0,n_r,m-1) \le \cdots \le R(n_0,n_r,1) \le r-1.
	\]

	For \(k\in\mathbb N\), set
	\begin{equation}\label{eq:Nk_def}
		N_k=N_k(n_0)
		=
		\left\lfloor
		\frac{
			km-1-n_0\gcd(m,\Lambda)
		}{
			\Lambda
		}
		\right\rfloor.
	\end{equation}
	
	Using the same monotonicity argument as in the previous subsections, one obtains the following existence criterion.
	
	\begin{prop}\label{prop:existence_Q}
		Assume
		\(
		n_0\gcd(m,\Lambda)<\Lambda \le m.
		\)
		If
		\(
		k\in\left\{1,\ldots,\frac{\Lambda}{2}-1\right\}
		\)
		and
		\(
		n_r=N_k>0,
		\)
		then \(\mathcal Q(n_0,n_r)\) holds.
	\end{prop}
	
	\subsubsection{Explicit Coefficients for \(\lambda_1=\cdots=\lambda_{r-1}=1\) and \(\lambda_r=2\)}
	
	We now determine explicitly the coefficients of the divisors obtained in Proposition~\ref{prop:existence_Q}. As in the previous subsections, the coefficients are recovered from the inequalities
	\(
	R(n_0,N_k,j)\ge r-i.
	\)
	
	\begin{theorem}\label{thm:explicit_coefficients_l=2}
		Let
		\(
		y^{m}
		=
		\prod_{i=1}^{r}(x-\alpha_i)^{\lambda_i}
		\)
		be a Kummer extension over \(\cK(x)\), where \(m\ge 2\) and \(r-1\ge 2\) are even,
		\(
		\operatorname{char}(\cK)\nmid m,
		\)
		the elements \(\alpha_i\in\cK\) are distinct,
		\(
		\lambda_1=\cdots=\lambda_{r-1}=1,
		\)
		and
		\(
		\lambda_r=2.
		\)
		Set \(\Lambda=r+1\).
		
		Let
		\[
		n_0\in
		\left\{
		0,\dots,
		\frac{\Lambda}{\gcd(m,\Lambda)}-1
		\right\},
		\qquad
		N_i
		=
		\left\lfloor
		\frac{
			mi-1-n_0\gcd(m,\Lambda)
		}{
			\Lambda
		}
		\right\rfloor.
		\]
		
		Assume \(n_0\gcd(m,\Lambda)<\Lambda \le m\), and let
		\(
		k\in\{1,\dots,\Lambda/2-1\}
		\)
		satisfy \(N_k>0\). Then
		\[
		A
		=
		n_0\frac{\operatorname{div}_{\infty}(x)}
		{m/\gcd(m,\Lambda)}
		+
		\sum_{i=1}^{r-1}n_iP_i
		+
		N_k\frac{\operatorname{div}_{0}(x-\alpha_r)}
		{m/2}
		\]
		is an effective divisor of degree \(g\) with \(\ell(A)=1\), where
		\[
		n_i
		=
		\begin{cases}
			\max(0,N_{i-1}),
			&
			1\le i\le k,
			\\
			N_i,
			&
			k<i<k+\Lambda/2,
			\\
			N_{i+1},
			&
			k+\Lambda/2\le i\le r-1.
		\end{cases}
		\]
	\end{theorem}
	
	\begin{proof}
		By Proposition~\ref{prop:existence_Q}, it is enough to compute $n_1,\ldots,n_{r-1}$. We may assume that $n_1\le\cdots\le n_{r-1}$.
		As in the previous cases, the coefficients are obtained from the inequalities
		\(
		R(n_0,N_k,j)\ge r-i.
		\)
		Rewriting these inequalities in terms of \(j\) yields
		\[
		j
		\le
		N_{\,i-[j\le N_k]+[j>N_{k+\Lambda/2}]}.
		\]
		
		Since \(m\ge\Lambda\), the sequence \((N_i)\) is strictly increasing. Hence:
		for \(1\le i\le k\), one obtains \(n_i=\max(0,N_{i-1})\);
		for \(k<i<k+\Lambda/2\), one gets \(n_i=N_i\);
		and for \(k+\Lambda/2\le i\le r-1\), one has \(n_i=N_{i+1}\).
	\end{proof}
	
	\section{Applications to LCP AG Codes}
	
	We now apply the non-special divisors constructed above to obtain families of linear complementary pair (LCP) AG codes over function fields of genus \(g\ge1\).
	
	Recall that two AG codes \(\cC_{\cL}(D,G)\) and \(\cC_{\cL}(D,H)\) form an LCP pair if \(\cC_{\cL}(D,G)\oplus\cC_{\cL}(D,H)=\mathbb F_q^n.\) From Theorem~\ref{prop:LCP_AG_Criterion}, this holds whenever \(\ell(G)+\ell(H)=n\), and the divisors \(\gcd(G,H)\) and \(\operatorname{lmd}(G,H)-D\) are both non-special divisors of degree \(g-1\).
	
	Let \(A\) be a non-special divisor of degree \(\deg(A)=g\) and \(Q_\infty\) be a rational place of \(\cF\). In this case, \(A-Q_\infty\) is non-special of degree \(g-1\) by Proposition~\ref{prop:transition_nonspecial}.
	
	Suppose there exists a function \(h \in \cF\) such that \(D = \operatorname{div}_{0}(h)\). For \(\Phi\in\cF\) and \(s\in\mathbb N\), define
	\begin{equation}\label{eq:divisor_G}
		G = A - Q_\infty + \operatorname{div}_{\infty}(h) - \operatorname{div}_{\infty}(\Phi^s),
	\end{equation}
	and
	\begin{equation}\label{eq:divisor_H}
		H = A - Q_\infty + \operatorname{div}_{0}(\Phi^s).
	\end{equation}
	
	Assume that \(\operatorname{div}_{\infty}(h) \ge \operatorname{div}_{\infty}(\Phi^s)\) and that the supports of \(\operatorname{div}_{0}(\Phi)\) and \(\operatorname{div}_{\infty}(h)\) are disjoint. Then
	\[
	\gcd(G,H)=A-Q_\infty,
	\]
	which is non-special of degree \(g-1\). Moreover, 
	\[
	\operatorname{lmd}(G,H) = A - Q_\infty + \operatorname{div}_{0}(\Phi^s) + \operatorname{div}_{\infty}(h) - \operatorname{div}_{\infty}(\Phi^s).
	\]
	
	Subtracting \(D = \operatorname{div}_{0}(h)\), we obtain
	\begin{align*}
		\operatorname{lmd}(G,H)-D &= A - Q_\infty + \bigl(\operatorname{div}_{0}(\Phi^s)-\operatorname{div}_{\infty}(\Phi^s)\bigr) - \bigl(\operatorname{div}_{0}(h)-\operatorname{div}_{\infty}(h)\bigr) \\
		&= A - Q_\infty + (\Phi^s) - (h) \\
		&\sim A-Q_\infty.
	\end{align*}
	Hence, \(\operatorname{lmd}(G,H)-D\) is also non-special of degree \(g-1\). Therefore, provided that the degree conditions are met to ensure \(\ell(G)+\ell(H)=n\),
	\[
	\cC_{\cL}(D,G) \quad\text{and}\quad \cC_{\cL}(D,H)
	\]
	form an LCP pair.
	
	In the next subsections, we specialize this construction to Kummer extensions using the explicit non-special divisors obtained in Section~3.
	\subsection{LCP Constructions via Half-Degree Ramification}
	
	We first consider Kummer extensions containing a branch place with ramification index \(m/2\). In contrast with previous constructions based only on totally ramified places (see, e.g., \cite{4author}), this approach incorporates non-totally ramified places into the construction of invariant non-special divisors.
	
	\begin{theorem}\label{thm:LCP_Kummer_General}
		Let
		\[
		y^m=\prod_{i=1}^r (x-\alpha_i)^{\lambda_i}
		\]
		define a Kummer extension \(\cF/\mathbb{F}_q(x)\), where \(m\ge4\) is even, \(\operatorname{char}(\mathbb{F}_q)\nmid m\), the elements \(\alpha_1,\ldots,\alpha_r\in\mathbb{F}_q\) are pairwise distinct, and
		\[
		\lambda_1=\cdots=\lambda_{r-1}=1,
		\qquad
		\lambda_r=\frac{m}{2}.
		\]
		Assume that
		\[
		\frac{m}{2}-\left(\frac{m}{2}\bmod2\right)
		\le r\le \frac{m}{2}.
		\]
		Let \(Q_1,\ldots,Q_{r-1}\) be the totally ramified places lying over
		\[
		x=\alpha_1,\ldots,x=\alpha_{r-1},
		\]
		let \(Q_\infty\) be a rational place at infinity of \(\cF\), and let \(E_\infty = \divv_\infty(x) \) be the pole divisor of \(x\) in \(\cF\).
		Suppose that \(a_1,\ldots,a_t\in\mathbb{F}_q\) are distinct from \(\alpha_1,\ldots,\alpha_r\) and that each place \(P_{a_i}\) splits completely in \(\cF/\mathbb{F}_q(x)\). Define
		\[
		D=\sum_{i=1}^t \ \sum_{P_{a_i,b}\mid P_{a_i}} P_{a_i,b},
		\]
		so that \(\deg(D)=n\), where \(n=tm\).
		For \(s\in\mathbb{N}\) satisfying
		\[
		\frac{g-1}{m(r-1)}
		<
		s
		<
		\frac{n-g+1}{m(r-1)},
		\]
		define
		\[
		G
		=
		\sum_{i=1}^{r-1} c_iQ_i
		+
		\bigl(t-s(r-1)\bigr)E_\infty
		-
		Q_\infty
		\]
		and
		\[
		H
		=
		\sum_{i=1}^{r-1}(c_i+sm)Q_i
		-
		Q_\infty,
		\]
		where
		\[
		c_i
		=
		2\left\lceil
		\frac{m(2i-1)-2(r-1)}{4(r-1)}
		\right\rceil
		-1.
		\]
		Then,
		\(
		\bigl(
		\cC_{\cL}(D,G),
		\cC_{\cL}(D,H)
		\bigr)
		\)
		is a linear complementary pair of AG codes over \(\mathbb{F}_q\).
	\end{theorem}
	
	\begin{proof}
		Set
		\(
		A=\sum_{i=1}^{r-1} c_iQ_i.
		\)
		By Theorem~\ref{thm:explicit coeff 1 of m/2}, $A$ is an effective non-special divisor of degree $g$; so
		\(
		A-Q_\infty
		\)
		is a non-special divisor of degree \(g-1\). 
		
		Let \(\Phi(x)=\prod_{i=1}^{r-1}(x-\alpha_i)\) and \(h(x)=\prod_{i=1}^t(x-a_i)\). 
		Because each \(Q_i\) is totally ramified and \(E_\infty\) is the pole divisor of \(x\) with \(\deg(E_\infty)=m\), we have
		\[
		\operatorname{div}_0(\Phi) = m\sum_{i=1}^{r-1}Q_i
		\quad\text{and}\quad
		\operatorname{div}_\infty(\Phi) = (r-1)E_\infty.
		\]
		Similarly, since every \(P_{a_i}\) splits completely, \(D=\operatorname{div}_0(h)\) and \(\operatorname{div}_\infty(h)=tE_\infty\). 
		
		Following the construction described in equations \eqref{eq:divisor_G} and \eqref{eq:divisor_H}, we can rewrite \(G\) and \(H\) as
		\begin{align*}
			G &= A - Q_\infty + \operatorname{div}_\infty(h) - \operatorname{div}_\infty(\Phi^s) \\
			&= \sum_{i=1}^{r-1} c_iQ_i - Q_\infty + tE_\infty - s(r-1)E_\infty \\
			&= \sum_{i=1}^{r-1} c_iQ_i + \bigl(t-s(r-1)\bigr)E_\infty - Q_\infty,
		\end{align*}
		and
		\begin{align*}
			H &= A - Q_\infty + \operatorname{div}_0(\Phi^s) \\
			&= \sum_{i=1}^{r-1} c_iQ_i - Q_\infty + sm\sum_{i=1}^{r-1}Q_i \\
			&= \sum_{i=1}^{r-1}(c_i+sm)Q_i - Q_\infty.
		\end{align*}
		
		Note that the supports of \(\operatorname{div}_0(\Phi)\) and \(\operatorname{div}_\infty(h)\) are disjoint. The upper bound on \(s\) guarantees that \(t - s(r-1) > \frac{g-1}{m} \ge 0\). Since \(t\) and \(s(r-1)\) are integers, \(t - s(r-1) \ge 1\), which implies \(\operatorname{div}_\infty(h) \ge \operatorname{div}_\infty(\Phi^s)\). Furthermore, 
		\(
		\operatorname{Supp}(D) \cap \operatorname{Supp}(G) = \operatorname{Supp}(D) \cap \operatorname{Supp}(H) = \emptyset,
		\)
		ensuring the AG codes are well defined.
		
		The degrees of \(H\) and \(G\) are
		\[
		\deg(H)= sm(r-1) + g - 1
		\]
		and
		\[
		\deg(G)= \bigl(t-s(r-1)\bigr)m + g - 1 = n - sm(r-1) + g - 1.
		\]
		The assumptions on \(s\) imply \(2g-2<\deg(G),\deg(H)<n\). By the Riemann--Roch theorem, it holds that \(\ell(G)+\ell(H)=n\).
		
		As established by the general method, since \(G\) and \(H\) are constructed from \(\Phi\), \(h\), and \(A\), it automatically follows that \(\gcd(G,H) = A - Q_\infty\) and \(\operatorname{lmd}(G,H) - D \sim A - Q_\infty\). Because \(A - Q_\infty\) is non-special of degree \(g-1\), Theorem~\ref{prop:LCP_AG_Criterion} concludes that \(\bigl(\cC_{\cL}(D,G), \cC_{\cL}(D,H)\bigr)\) is an LCP pair of AG codes over \(\mathbb{F}_q\).
	\end{proof}
	
	We now specialize the previous construction to Kummer extensions associated with Dickson polynomials of the first kind.

	\begin{prop}
		Let \(\cF/\mathbb{F}_{q^2}\) be defined by
		\[
		y^m=(x+2)^{m/2}\varphi_{\frac{m-2}{2}}(x),
		\]
		where \(m\ge4\) is even, $q\equiv m-1 \pmod{m(m-2)}$,
		and \(\varphi_{\frac{m-2}{2}}(x)\) denotes the Dickson polynomial of the first kind. Let $Q_1,\ldots,Q_{(m-2)/2}$ be the totally ramified places corresponding to the roots of $\varphi_{\frac{m-2}{2}}(x)$.
		Let \(Q_\infty\) be a rational place of \(\cF\), and let \(E_\infty = \divv_\infty(x)\) be the pole divisor of \( x \) in \(\cF\).
		
		Suppose that \(D\) is a divisor of degree \(n=mt\) supported on completely split rational places. Let \(s\in\mathbb{N}\) satisfy
		\[
		\frac{2(g-1)}{m(m-2)}
		<
		s
		<
		\frac{2(n-g+1)}{m(m-2)}.
		\]
		Define
		\begin{align*}
			G
			&=
			\sum_{i=1}^{(m-2)/2}
			\left(
			2\left\lceil
			\frac{m(i-1)+1}{m-2}
			\right\rceil
			-1
			\right)Q_i
			+
			\left(
			t-s\frac{m-2}{2}
			\right)E_\infty
			-
			Q_\infty,
			\\[1ex]
			H
			&=
			\sum_{i=1}^{(m-2)/2}
			\left(
			2\left\lceil
			\frac{m(i-1)+1}{m-2}
			\right\rceil
			-1+sm
			\right)Q_i
			-
			Q_\infty.
		\end{align*}
		Then, \(
		\bigl(
		\cC_{\cL}(D,G),
		\cC_{\cL}(D,H)
		\bigr)
		\) is an LCP pair of AG codes over \(\mathbb{F}_{q^2}\) with parameters
		\[
		\cC_{\cL}(D,G):
		\left[
		mt,\
		n-s\frac{m(m-2)}{2},\
		\ge
		s\frac{m(m-2)}{2}+1-g
		\right]
		\]
		and
		\[
		\cC_{\cL}(D,H):
		\left[
		mt,\
		s\frac{m(m-2)}{2},\
		\ge
		n-s\frac{m(m-2)}{2}+1-g
		\right].
		\]
	\end{prop}
	
	\begin{proof}
		Note that the roots of the Dickson polynomial \(\varphi_{\frac{m-2}{2}}(x)\) correspond to the elements \(\alpha_1, \ldots, \alpha_{r-1}\). As established in \cite{Kazemifard2018}, under the condition \(q\equiv m-1 \pmod{m(m-2)}\), \(\varphi_{\frac{m-2}{2}}(x)\) has exactly \((m-2)/2\) distinct simple roots in \(\mathbb{F}_{q^2}\), and these roots are distinct from \(\alpha_r = -2\). Thus, the elements \(\alpha_1, \ldots, \alpha_r\) are pairwise distinct. Note also that \(r = m/2\) trivially satisfies the bounds
		\[
		\frac{m}{2}-\left(\frac{m}{2}\bmod2\right) \le r\le \frac{m}{2}.
		\]
		Hence, the Kummer extension \(\cF/\mathbb{F}_{q^2}\) given by
		\[
		y^m=(x+2)^{m/2}\varphi_{\frac{m-2}{2}}(x)
		\]
		fits the hypothesis of Theorem~\ref{thm:LCP_Kummer_General} with \(r = m/2\).
		
		Substituting \(r-1 = (m-2)/2\) into the formulas from Theorem~\ref{thm:LCP_Kummer_General}, the condition on \(s\) becomes exactly
		\[
		\frac{2(g-1)}{m(m-2)} < s < \frac{2(n-g+1)}{m(m-2)}.
		\]
		Furthermore, the coefficients \(c_i\) simplify to
		\begin{align*}
			c_i 
			&= 2\left\lceil \frac{m(2i-1)-2(r-1)}{4(r-1)} \right\rceil -1 \\
			&= 2\left\lceil \frac{2m(i-1) + m - (m-2)}{2(m-2)} \right\rceil -1 \\
			&= 2\left\lceil \frac{m(i-1)+1}{m-2} \right\rceil -1,
		\end{align*}
		and the defined divisors \(G\) and \(H\) match the exact construction in Theorem~\ref{thm:LCP_Kummer_General}.
		Thus, by Theorem~\ref{thm:LCP_Kummer_General}, \(\bigl(\cC_{\cL}(D,G), \cC_{\cL}(D,H)\bigr)\) forms an LCP pair. Finally, since 
		\[
		\deg(G) = n - s\frac{m(m-2)}{2} + g - 1 \quad \mbox{and} \quad \deg(H) = s\frac{m(m-2)}{2} + g - 1,
		\]
		the code parameters follow directly from the Riemann--Roch theorem and the standard AG code bounds.
	\end{proof}

	\subsection{LCP Pairs via Balanced High-Order Ramification}
	
	We next consider Kummer extensions whose defining polynomial contains two branch places of ramification index $m/2$. This symmetric ramification profile enlarges the class of admissible invariant divisors and provides additional flexibility in the construction of LCP AG codes with competitive parameters.
	
	From a computational perspective, the simultaneous presence of multiple non-totally ramified places substantially complicates traditional approaches based on Weierstrass semigroups and gap sequences. In contrast, the invariant divisor method developed in this paper remains effective under such mixed ramification behavior and yields explicit constructions in closed form.
	
	The following theorem provides an explicit family of LCP AG codes arising from this balanced high-order ramification configuration.
	
	\begin{theorem}\label{thm:LCP_Kummer_l=m/2_N=1}
		Let $\mathcal{F}/\mathbb{F}_{q}$ be the Kummer extension defined by
		\[
		y^m=\prod_{i=1}^{r}(x-\alpha_i)^{\lambda_i},
		\]
		where $m\ge 4$ is an even integer such that
		$\operatorname{char}(\mathbb{F}_q)\nmid m$, and
		$\alpha_1,\ldots,\alpha_r\in\mathbb{F}_q$ are distinct.
		Assume that
		\[
		\lambda_1=\cdots=\lambda_{r-2}=1,
		\qquad
		\lambda_{r-1}=\lambda_r=\frac{m}{2}.
		\]
		
		Let $Q_1,\ldots,Q_{r-2}$ be the totally ramified places corresponding to
		$x=\alpha_1,\ldots,x=\alpha_{r-2}$. Let $Q_\infty$ be a rational place of $\mathcal{F}$, and let $E_\infty = (x)_\infty$ be the pole divisor of $x$ in $\mathcal{F}$.
		
		Let $a_1,\ldots,a_t\in\mathbb{F}_q$ be distinct from $\alpha_1,\ldots,\alpha_r$ such that the places $P_{a_i}$ split completely in the extension $\mathcal{F}/\mathbb{F}_q(x)$. Define
		\[
		D=
		\sum_{i=1}^{t}
		\sum_{P_{a_i,b}\mid P_{a_i}}
		P_{a_i,b},
		\]
		so that $\deg(D)=n$, where $n=mt$. For an integer $s$ satisfying
		\[
		\frac{g-1}{m\left(\frac{m}{2}-1\right)}
		<
		s
		<
		\frac{n-g+1}{m\left(\frac{m}{2}-1\right)},
		\]
		define
		\[
		G=
		\sum_{i=1}^{m/2-1}(2i-1)Q_i
		+\frac{1}{2}\operatorname{div}_{0}(x-\alpha_{r-1})
		+\left(t-s\left(\frac{m}{2}-1\right)\right)E_\infty
		-Q_\infty,
		\]
		and
		\[
		H=
		\sum_{i=1}^{m/2-1}(2i-1+sm)Q_i
		+\frac{1}{2}\operatorname{div}_{0}(x-\alpha_{r-1})
		-Q_\infty.
		\]
		Then, $\left(
		\mathcal{C}_{\mathcal{L}}(D,G),
		\mathcal{C}_{\mathcal{L}}(D,H)
		\right)$ is a linear complementary pair of AG codes over $\mathbb{F}_q$.
	\end{theorem}
	
	\begin{proof}
		Let $\Phi(x)=\prod_{i=1}^{r-2}(x-\alpha_i)\in \mathcal{F}$. Since the places $Q_1,\ldots,Q_{r-2}$ are totally ramified, $r-2=\frac{m}{2}-1$, and $E_\infty$ is the pole divisor of $x$ with $\deg(E_\infty)=m$, the principal divisor of $\Phi$ satisfies
		\begin{equation}
			\label{eq:phi_s_N1}
			\operatorname{div}_{0}(\Phi) = m\sum_{i=1}^{m/2-1}Q_i,
			\qquad
			\operatorname{div}_{\infty}(\Phi) = \left(\frac{m}{2}-1\right)E_\infty .
		\end{equation}
		
		Define the base divisor $A$ as
		\begin{equation}
			\label{eq:base_A_N1}
			A=
			\sum_{i=1}^{m/2-1}(2i-1)Q_i
			+
			\frac{1}{2}\operatorname{div}_{0}(x-\alpha_{r-1}).
		\end{equation}
		By Theorem \ref{thm:explicit coeff 2 of m/2}, $A$ is a non-special divisor of degree $g$, so $A-Q_\infty$ is non-special of degree $g-1$. 
		
		Consider $h(x) = \prod_{i=1}^t(x-a_i)$. Since the rational places $P_{a_i}$ split completely, we have $D = \operatorname{div}_0(h)$ and $\operatorname{div}_\infty(h) = tE_\infty$. Using the general construction introduced at the beginning of this section, we construct $G$ and $H$ as
		\begin{align*}
			G &= A - Q_\infty + \operatorname{div}_{\infty}(h) - \operatorname{div}_{\infty}(\Phi^s) \\
			&= A - Q_\infty + tE_\infty - s\left(\frac{m}{2}-1\right)E_\infty, \\[1ex]
			H &= A - Q_\infty + \operatorname{div}_{0}(\Phi^s) \\
			&= A - Q_\infty + sm\sum_{i=1}^{m/2-1}Q_i.
		\end{align*}
		Substituting \eqref{eq:base_A_N1} into these expressions yields exactly the divisors stated in the theorem.
		
		By the general framework, since the condition on $s$ guarantees $t - s(\frac{m}{2}-1) \ge 1$, we automatically obtain
		\[
		\gcd(G,H)=A-Q_\infty,
		\]
		and
		\[
		\operatorname{lmd}(G,H)-D \sim A-Q_\infty .
		\]
		Since $A-Q_\infty$ is non-special of degree $g-1$, both divisors $\gcd(G,H)$ and $\operatorname{lmd}(G,H)-D$ satisfy the LCP criterion. Finally, the bounds on $s$ ensure that $2g-2<\deg(G),\deg(H)<n$. Therefore, $(\mathcal{C}_{\mathcal{L}}(D,G), \mathcal{C}_{\mathcal{L}}(D,H))$ forms a linear complementary pair of AG codes over $\mathbb{F}_q$.
	\end{proof}
	
	\begin{theorem}\label{thm:LCP_Kummer_Two_Multiple_Roots_N2}
		Let $\mathcal{F}/\mathbb{F}_{q}$ be the Kummer extension defined by
		\[
		y^m
		=
		\prod_{i=1}^{m+1}(x-\alpha_i)
		(x-\alpha_{m+2})^{m/2}
		(x-\alpha_{m+3})^{m/2},
		\]
		where $m\ge 4$ is an even integer such that $\operatorname{char}(\mathbb{F}_q)\nmid m$. Let $Q_1,\ldots,Q_{m+1}$ be the totally ramified places corresponding to $x=\alpha_1,\ldots,x=\alpha_{m+1}$. Let $Q_\infty$ be a rational place of $\mathcal{F}$, and let $E_\infty = (x)_\infty$ be the pole divisor of $x$ in $\mathcal{F}$.
		
		Let $c_k$ be the coefficients determined in Theorem~\ref{thm:explicit coeff 2 of m/2}, satisfying
		\[
		\sum_{k=1}^{m+1} c_k Q_k
		=
		\sum_{i=1}^{m/2-1}
		2i\,(Q_{2i-1}+Q_{2i}).
		\]
		
		Let $a_1,\ldots,a_t\in\mathbb{F}_q$ be distinct from $\alpha_1,\ldots,\alpha_{m+3}$ such that the places $P_{a_j}$ split completely in the extension $\mathcal{F}/\mathbb{F}_q(x)$. Define
		\[
		D=
		\sum_{j=1}^{t}
		\sum_{P_{a_j,b}\mid P_{a_j}}
		P_{a_j,b},
		\]
		so that $\deg(D)=n$, where $n=mt$.
		
		For an integer $s$ satisfying
		\[
		\frac{g-1}{m(m+1)}
		<
		s
		<
		\frac{n-g+1}{m(m+1)},
		\]
		define
		\[
		G=
		\sum_{k=1}^{m+1} c_k Q_k
		+\frac{1}{2}\operatorname{div}_{0}(x-\alpha_{m+2})
		+\frac{1}{2}\operatorname{div}_{0}(x-\alpha_{m+3})
		+\bigl(t-s(m+1)\bigr)E_\infty - Q_\infty,
		\]
		and
		\[
		H=
		\sum_{k=1}^{m+1} (c_k+sm)Q_k
		+\frac{1}{2}\operatorname{div}_{0}(x-\alpha_{m+2})
		+\frac{1}{2}\operatorname{div}_{0}(x-\alpha_{m+3})
		-Q_\infty.
		\]
		
		Then
		\[
		\left(
		\mathcal{C}_{\mathcal{L}}(D,G),
		\mathcal{C}_{\mathcal{L}}(D,H)
		\right)
		\]
		is a linear complementary pair of AG codes over $\mathbb{F}_q$.
	\end{theorem}
	
	\begin{proof}
		We apply the established LCP structural framework. Let $\Phi(x) = \prod_{i=1}^{m+1}(x-\alpha_i) \in \mathcal{F}$ and $h(x) = \prod_{j=1}^t(x-a_j) \in \mathcal{F}$. Because the rational places $P_{a_j}$ split completely, $D = \operatorname{div}_0(h)$ and $\operatorname{div}_\infty(h) = tE_\infty$. The principal divisor of $\Phi$ is characterized by $\operatorname{div}_0(\Phi) = m \sum_{k=1}^{m+1} Q_k$ and $\operatorname{div}_\infty(\Phi) = (m+1)E_\infty$.
		
		Define the base divisor $A$ as
		\begin{equation*}
			A = \sum_{k=1}^{m+1} c_k Q_k + \frac{1}{2}\operatorname{div}_0(x-\alpha_{m+2}) + \frac{1}{2}\operatorname{div}_0(x-\alpha_{m+3}).
		\end{equation*}
		By Theorem~\ref{thm:explicit coeff 2 of m/2}, the coefficients $c_k$ guarantee that $A - Q_\infty$ is a non-special divisor of degree $g-1$. 
		
		Using the general formulation, the divisors defined in the theorem exactly satisfy $G = A - Q_\infty + \operatorname{div}_\infty(h) - \operatorname{div}_\infty(\Phi^s)$ and $H = A - Q_\infty + \operatorname{div}_0(\Phi^s)$. Thus, the properties of the coordinate-wise operators directly yield $\gcd(G,H) = A - Q_\infty$ and $\mathrm{lmd}(G,H) - D \sim A - Q_\infty$. 
		
		Since both $\gcd(G,H)$ and $\mathrm{lmd}(G,H) - D$ are non-special divisors of degree $g-1$, and the conditions on $s$ ensure $2g-2 < \deg(G), \deg(H) < n$, all structural conditions are satisfied. Thus, $(\mathcal{C}_{\mathcal{L}}(D,G), \mathcal{C}_{\mathcal{L}}(D,H))$ is an LCP.
	\end{proof}
	
	\begin{prop}
		Let $\cF/\mathbb{F}_{q^2}$ be the function field defined by
		\[
		y^m = (x^2-4)^{m/2}\varphi_{m+1}(x),
		\]
		where $m \ge 4$ is even, $\operatorname{char}(\mathbb{F}_{q^2}) \nmid m(m+1)$, and $\varphi_{m+1}(x)$ denotes the Dickson polynomial of the first kind. Let $Q_1,\dots,Q_{m+1}$ be the totally ramified places corresponding to the roots of $\varphi_{m+1}(x)$. Let $Q_\infty$ be a rational place of $\mathcal{F}$, and let $E_\infty = (x)_\infty$ be the pole divisor of $x$ in $\mathcal{F}$.
		
		Let $D$ be a divisor of degree $n=mt$ supported on completely split rational places. For an integer $s$ satisfying
		\[
		\frac{g-1}{m(m+1)}
		<
		s
		<
		\frac{n-g+1}{m(m+1)},
		\]
		define
		\begin{align*}
			G &=
			\sum_{k=1}^{m+1} c_k Q_k
			+ \frac{1}{2}\operatorname{div}_0(x-2)
			+ \frac{1}{2}\operatorname{div}_0(x+2)
			+ \bigl(t-s(m+1)\bigr)E_\infty - Q_\infty, 
		\end{align*}
		and
		\begin{align*}
			H &=
			\sum_{k=1}^{m+1} (c_k+sm)Q_k
			+ \frac{1}{2}\operatorname{div}_0(x-2)
			+ \frac{1}{2}\operatorname{div}_0(x+2)
			- Q_\infty,
		\end{align*}
		where the coefficients $c_k$ are given in Theorem~\ref{thm:explicit coeff 2 of m/2}. Then
		\[
		(\mathcal{C}_{\mathcal{L}}(D,G),
		\mathcal{C}_{\mathcal{L}}(D,H))
		\]
		is an LCP of AG codes over $\mathbb{F}_{q^2}$ with parameters
		\begin{itemize}
			\item
			\(
			\mathcal{C}_{\mathcal{L}}(D,G):
			[mt,\; n-sm(m+1),\;
			\ge sm(m+1)+1-g], \mbox{ and}
			\)
			
			\item
			\(
			\mathcal{C}_{\mathcal{L}}(D,H):
			[mt,\; sm(m+1),\;
			\ge n-sm(m+1)+1-g].
			\)
		\end{itemize}
	\end{prop}
	
	\begin{proof}
		The function field satisfies the hypotheses of Theorem~\ref{thm:LCP_Kummer_Two_Multiple_Roots_N2}. Indeed,
		\[
		(x^2-4)^{m/2}
		=
		(x-2)^{m/2}(x+2)^{m/2},
		\]
		so the defining polynomial contains exactly two roots of multiplicity $m/2$, namely $2$ and $-2$. Moreover, the Dickson polynomial \(\varphi_{m+1}(x)\) has \(m+1\) distinct simple roots, and by \cite{Kazemifard2018}, the corresponding function field is maximal over \(\mathbb{F}_{q^2}\).
		
		The parameters and structure match perfectly with Theorem~\ref{thm:LCP_Kummer_Two_Multiple_Roots_N2}, where \(\alpha_{m+2} = 2\) and \(\alpha_{m+3} = -2\). By applying the general method introduced in the theorem, both $\gcd(G,H)$ and $\operatorname{lmd}(G,H)-D$ are linearly equivalent to $A - Q_\infty$, which is non-special of degree $g-1$. Therefore, the pair forms an LCP of AG codes over $\mathbb{F}_{q^2}$.
		
		Finally, since $\deg(E_\infty) = m$, the degrees of $G$ and $H$ are exactly $n-sm(m+1)+g-1$ and $sm(m+1)+g-1$, respectively. The conditions imposed on $s$ ensure that $2g-2 < \deg(G),\deg(H) < n$. By the Riemann--Roch theorem, the dimensions evaluate to $k_G=n-sm(m+1)$ and $k_H=sm(m+1)$, and the stated lower bounds on the minimum distances follow directly from the Goppa bound.
	\end{proof}
	
	\subsection{LCP Pairs via Minimal Deviation from Total Ramification}
	
	We now consider Kummer extensions containing a single branch place with ramification index two. Although this represents only a mild deviation from the totally ramified setting, the corresponding divisor structure becomes substantially more intricate, particularly in the characterization of non-special divisors of degrees $g$ and $g-1$.
	
	Methods based on explicit Weierstrass semigroup computations are typically difficult to apply in this setting, since the presence of low-order ramification alters the combinatorial structure of the associated gap sequences. In contrast, the invariant divisor approach developed in this paper remains applicable without requiring a complete description of the Weierstrass semigroup. This yields explicit constructions of LCP AG codes for a broader class of Kummer extensions.
	\begin{theorem}\label{thm:LCP_Kummer_Lambda2}
		Let $\cF/\mathbb{F}_{q}$ be the Kummer extension defined by
		\[
		y^m=(x-\alpha_r)^2\prod_{i=1}^{r-1}(x-\alpha_i),
		\]
		where $\alpha_1,\dots,\alpha_r\in\mathbb{F}_{q}$ are distinct, $m\ge r+1$, both $m$ and $r-1$ are even, and $\operatorname{char}(\mathbb{F}_{q})\nmid m$. Let $P_i$ $(1\le i\le r-1)$ denote the totally ramified place lying over $x=\alpha_i$. Let $Q_\infty$ be a rational place of $\mathcal{F}$, and let $E_\infty = (x)_\infty$ be the pole divisor of $x$ in $\mathcal{F}$.
		
		Choose $t$ elements
		\[
		a_1,\dots,a_t
		\in
		\mathbb{F}_{q}\setminus
		\{\alpha_1,\dots,\alpha_r\}
		\]
		such that each corresponding place of $\mathbb{F}_{q}(x)$ splits completely in $\cF/\mathbb{F}_{q}(x)$. Let
		\[
		D=
		\sum_{i=1}^{t}
		\sum_{P_{a_i,b}\mid P_{a_i}}
		P_{a_i,b}
		\]
		be the associated divisor of degree $n$, where $n=tm$. Define
		\[
		A=
		\sum_{i=1}^{r-1} n_i P_i
		+
		N_k\frac{2}{m}\operatorname{div}_0(x-\alpha_r),
		\]
		where $A$ is the non-special divisor of degree $g$ obtained in Theorem~\ref{thm:explicit_coefficients_l=2} for the case $Z=0$. For any integer $s$ satisfying
		\[
		\frac{g-1}{m(r-1)}
		<
		s
		<
		\frac{n-g+1}{m(r-1)},
		\]
		define
		\[
		G=
		A+
		\bigl(t-s(r-1)\bigr)E_\infty
		-
		Q_\infty
		\]
		and
		\[
		H=
		A+
		sm\sum_{i=1}^{r-1}P_i
		-
		Q_\infty.
		\]
		Then,
		\(
		(\mathcal{C}_{\mathcal{L}}(D,G),
		\mathcal{C}_{\mathcal{L}}(D,H))
		\) forms an LCP of AG codes over $\mathbb{F}_{q}$.
	\end{theorem}
	
	\begin{proof}
		By Theorem~\ref{thm:explicit_coefficients_l=2}, the divisor $A$ is non-special of degree $g$. Hence, $A-Q_\infty$ is a non-special divisor of degree $g-1$.
		
		Let $\Phi(x)=\prod_{i=1}^{r-1}(x-\alpha_i) \in \mathcal{F}$ and $h(x)=\prod_{i=1}^{t}(x-a_i) \in \mathcal{F}$. Since each place $P_i$ is totally ramified, and $E_\infty$ is the pole divisor of $x$ with $\deg(E_\infty)=m$, the principal divisor of $\Phi$ is given by
		\[
		\operatorname{div}_0(\Phi) = m\sum_{i=1}^{r-1}P_i
		\quad\text{and}\quad
		\operatorname{div}_\infty(\Phi) = (r-1)E_\infty.
		\]
		Similarly, because each place corresponding to $a_i$ splits completely in $\cF/\mathbb{F}_{q}(x)$, we have $D = \operatorname{div}_0(h)$ and $\operatorname{div}_\infty(h) = tE_\infty$. Using the general construction established at the beginning of this section, we can write $G$ and $H$ exactly as:
		\begin{align*}
			G &= A - Q_\infty + \operatorname{div}_\infty(h) - \operatorname{div}_\infty(\Phi^s) \\
			&= A - Q_\infty + tE_\infty - s(r-1)E_\infty \\
			&= A + \bigl(t-s(r-1)\bigr)E_\infty - Q_\infty, 
		\end{align*}
		and 
		\begin{align*} 
			H &= A - Q_\infty + \operatorname{div}_0(\Phi^s) \\
			&= A - Q_\infty + sm\sum_{i=1}^{r-1}P_i.
		\end{align*}
		The bounds on $s$ ensure that $t-s(r-1) \ge 1$, which means $\operatorname{div}_\infty(h) \ge \operatorname{div}_\infty(\Phi^s)$. By the general LCP framework, it automatically follows that 
		\[
		\gcd(G,H) = A - Q_\infty
		\]
		and
		\[
		\operatorname{lmd}(G,H)-D \sim A - Q_\infty.
		\]
		Since $A - Q_\infty$ is non-special of degree $g-1$, both divisors satisfy the non-special criterion. 
		
		Finally, the assumptions on $s$ guarantee that $2g-2 < \deg(G), \deg(H) < n$. Applying the Riemann--Roch theorem yields $\ell(G) = \deg(G)+1-g$ and $\ell(H) = \deg(H)+1-g$. Thus, $\ell(G)+\ell(H)=n$, satisfying all necessary conditions from the LCP criterion. Therefore, $(\mathcal{C}_{\mathcal{L}}(D,G), \mathcal{C}_{\mathcal{L}}(D,H))$ forms a linear complementary pair of AG codes over $\mathbb{F}_{q}$.
	\end{proof}
	
	\begin{example}
		\rm{Consider the function field
			\(
			y^8=x^2(x^4+1)
			\)
			over $\mathbb{F}_{49}$. In the notation of \cite{Tafazolian2019}, this function field corresponds to $\mathcal{X}(8,2,4)$. Since
			\[
			\gcd(q,nm)=\gcd(7,32)=1,
			\]
			the function field is maximal over $\mathbb{F}_{49}$. Here, the extension degree is $m=8$, the place $x=0$ corresponds to a root of multiplicity two, and the polynomial $x^4+1$ splits completely over $\mathbb{F}_{49}$, yielding four simple roots.
			
			By \cite[Lemma~2.1]{Tafazolian2019}, the genus of the function field is
			\(
			g=13,
			\)
			and the number of rational places attains the Hasse--Weil bound:
			\[
			N=7^2+1+2(7)(13)=232.
			\]
			
			To construct the divisor $A$, we apply Theorem~\ref{thm:explicit_coefficients_l=2} with $n_0=0$ and $k=1$. The sequence
			\[
			N_i=\left\lfloor \frac{8i-1}{6}\right\rfloor
			\]
			gives
			\[
			(N_1,N_2,N_3,N_4,N_5)=(1,2,3,5,6).
			\]
			The corresponding coefficients are
			\[
			(n_1,n_2,n_3,n_4)=(0,2,3,6).
			\]
			Hence,
			\[
			A=2P_2+3P_3+6P_4+D_5,
			\]
			where $D_5$ denotes the sum of the two ramified places lying over $x=0$. The divisor $A$ is a non-special divisor of degree $g=13$.
			
			To define the evaluation divisor $D$, we exclude the ramified places from the set of rational places. There are four places above the simple roots, two places above $x=0$, and two places at infinity. Therefore,
			\(
			232-8=224
			\) 
			rational places remain, corresponding to
			\(
			t=28
			\)
			completely split places. Consequently,
			\(
			n=8t=224.
			\)
			
			The admissible values of $s$ satisfy
			\[
			\frac{g-1}{m(r-1)}
			<
			s
			<
			\frac{n-g+1}{m(r-1)},
			\]
			that is,
			\[
			\frac{12}{32}<s<\frac{212}{32}.
			\]
			Choosing $s=2$, we obtain
			\begin{align*}
				G&=A+20D_\infty-Q_\infty,\\
				H&=A+16\sum_{i=1}^{4}P_i-Q_\infty.
			\end{align*}
			
			Their degrees are
			\(
			\deg(G)=13+20(8)-1=172
			\)
			and
			\(
			\deg(H)=13+16(4)-1=76.
			\)
			Applying the Riemann--Roch theorem yields
			\[
			k_G=160
			\quad \text{and} \quad
			k_H=64.
			\]
			Moreover, the Goppa bound gives
			\(
			d_G\ge 52
			\quad \text{and} \quad
			d_H\ge 148.
			\)
			
			Therefore,
			\(
			(\mathcal{C}_{\mathcal{L}}(D,G),
			\mathcal{C}_{\mathcal{L}}(D,H))
			\)
			is an LCP of AG codes over $\mathbb{F}_{49}$ with parameters
			\(
			[224,160,\ge 52]
			\)
			and
			\(
			[224,64,\ge 148].
			\)}
	\end{example}

	\section{Conclusion}
	
	In this paper, we developed a general framework for constructing non-special divisors in Kummer extensions beyond the classical setting of totally ramified places. The proposed approach is based on Galois-invariant divisor techniques and provides explicit criteria for identifying non-special divisors under general ramification configurations. In particular, the method remains effective in situations where traditional approaches based on Weierstrass semigroups become computationally difficult or inapplicable.
	
	As an application, we derived several explicit families of linear complementary pair (LCP) AG codes from Kummer function fields with non-standard ramification profiles. The constructions include cases with half-degree ramification, balanced multiple ramification places, and low-order ramification indices. For each family, explicit divisor descriptions and code parameters were obtained. The resulting codes provide flexible choices of dimensions and minimum-distance bounds while preserving the complementary structure required in cryptographic and side-channel resistant applications.
	
	The results demonstrate that invariant divisor methods offer a systematic and computationally efficient alternative for constructing AG codes from function fields with mixed ramification behavior. Future work may investigate analogous constructions for other classes of algebraic function fields, including Artin--Schreier extensions, recursive towers, and more general Galois covers.
	
	\section*{Acknowledgements}
	Saeed Tafazolian  was partially supported by CNPq grant no. 302774/2025-4, FAEPEX grant no. 3485/25, and FAPESP grant no. 2024/00923-6.
	
	Yuri da Silva was supported by CAPES: “This study was financed in part by the Coordenação de Aperfeiçoamento de Pessoal de Nível Superior - Brasil (CAPES) - Finance Code 001”.

\end{document}